\documentclass[11pt]{amsart}
\usepackage[english]{babel}
\usepackage{a4wide,color,graphicx,amsmath,amssymb,verbatim,mathrsfs,latexsym}
 \usepackage[colorlinks=true]{hyperref}
\allowdisplaybreaks
\newtheorem{theo}{Theorem}
\newtheorem{prop}{Proposition}
\newtheorem{lemma}{Lemma}

\newtheorem{cor}{Corollary}

\newcommand{\R}{\mathbb{R}}	
\newcommand{\eps}{\varepsilon}	
\newcommand{\pa}{\partial}		
\newcommand{\Div}{\textrm{div}\,}	

\newcommand{\na}{\nabla}		
	
\def\dive{\textnormal{div}}

\newcommand{\IRd}{\int_{\R^3}}
\newcommand{\IIRd}{\iint_{\R^3\times\R^3}}

\newcommand{\mass}{\|u_0\|_{L^1}}

\title{A review for an isotropic Landau model}
\author{Maria Gualdani and Nicola Zamponi}
\address{Department of Mathematics, George Washington University, 801 22nd Street NW, 20052 Washington DC, USA. }
\email{gualdani@gwu.edu}

\address{Institute for Analysis and Scientific Computing, Vienna University of  
	Technology, Wiedner Hauptstra\ss e 8--10, 1040 Wien, Austria.}
	
\email{nicola.zamponi@tuwien.ac.at}
\date{\today}

\begin{document}

\thanks{MPG is supported by NSF DMS-1514761. MPG would like to thank NCTS Mathematics Division Taipei for their kind hospitality. 
NZ acknowledges support from the Austrian Science Fund (FWF), grants P22108, P24304, W1245.} 

\begin{abstract}
%
%
%
We consider the equation 
$$
u_t  = \Div(a[u]\na u - u\na a[u]),\qquad -\Delta a = u.
$$
This model has attracted some attention in the recents years and several results are available in the literature. We review recent results on existence and smoothness of solutions and explain the open problems. 
\end{abstract}

\maketitle

\pagestyle{headings}		

\markboth{A review for an isotropic Landau model}{M. Gualdani, N. Zamponi}

\section{Introduction} 

\subsection{The isotropic Landau equation}
In this manuscript we review recent results on the isotropic Landau equation
\begin{equation}\label{landau}
\begin{array}{ll}
               u_t  = \Div(a[u]\na u - u\na a[u]),\qquad -\Delta a = u\qquad \mbox{in }\R^3,~~ t>0, \\
               u(\cdot,0)  =u_0.
             \end{array} 
\end{equation}

This problem has been extensively studied in the recent years. Due to its similarity to the semilinear heat equation, to the Keller-Segel model but mostly to the Landau, the analysis of existence, uniqueness and regularity of solutions to (\ref{landau}) is a very interesting problem. A modification of (\ref{landau}) was first introduced in \cite{KS,GKS}; there the authors studied existence and regularity of bounded radially symmetric and monotone decreasing solutions to 
$$
u_t  =a[u]\Delta u + \alpha u^2,  \quad \alpha \in \left(0,\frac{74}{75}\right).
$$

Existence of global bounded solutions for (\ref{landau}) has been proven in \cite{GG} when initial data are radially symmetric and monotone decreasing. Section \ref{rad_symme} explains these results more in details. Existence of weak solutions for even initial data has been shown in \cite{GZ}. See Section \ref{even_sol} for more details. 

For general initial data the problem of global existence of regular solutions is still open. The main obstacles for the analysis are hidden in the quadratic non-linearity: expanding the divergence term one can formally rewrite (\ref{landau}) as 
$$
u_t  =a[u]\Delta u + u^2. 
$$
This problem is reminiscent to the semilinear heat equation, which solutions become unbounded after a finite time  \cite{GK}.

\subsection{ Conserved quantities and entropy structure. } \label{subsec_entropy}
In this section we collect some properties of (\ref{landau}). The isotropic Landau equation shares some of the conservation properties of the classical Landau and Boltzmann equation. We first note that potential $a[u]$ can be expressed as
$$ a(x,t) = \IRd\frac{u(y,t)}{4\pi|x-y|}dy,\qquad x\in\R^3,~~t>0, $$
and therefore (\ref{landau}) can also be written as
\begin{align}
& u_t = \Div\IRd\frac{u(y)\na u(x) - u(x)\na u(y)}{4\pi|x-y|}dy.
\end{align}
With this in mind let us define the Maxwell-Boltzmann entropy:
\begin{align}
H[u] \equiv\IRd u \log u\, dx. \label{H}
\end{align}
The function $t\in(0,\infty)\mapsto H[u(t)]\in\R$ is nonincreasing in time: using \eqref{landau} we can write the entropy production as 
\begin{align*}
-4\pi\frac{d}{dt}H[u] &= \IIRd\frac{\na u(x)}{u(x)}\cdot\frac{u(y)\na u(x) - u(x)\na u(y)}{|x-y|}dx dy\\
&= \IIRd\frac{u(x)u(y)}{|x-y|}\frac{\na u(x)}{u(x)}\cdot\left( \frac{\na u(x)}{u(x)} - \frac{\na u(y)}{u(y)} \right)dx dy\\
&= \frac{1}{2}\IIRd\frac{u(x)u(y)}{|x-y|}\left| \frac{\na u(x)}{u(x)} - \frac{\na u(y)}{u(y)} \right|^2 dx dy\geq 0.
\end{align*}
Clearly $\IRd u(x,t)dx = \IRd u_0(x) dx$, $t>0$. We can say something about the first and second order moments of $u$.
From \eqref{landau} it follows
\begin{align*}
4\pi\frac{d}{dt}\IRd x u(x,t)dx = -\IIRd\frac{u(y)\na u(x) - u(x)\na u(y)}{|x-y|}dx dy = 0
\end{align*}
for obvious symmetry reasons. So the first moment is conserved. As for the second moment 
\begin{align*}
4\pi\frac{d}{dt}\IRd\frac{|x|^2}{2}u(x,t)dx &=  -\IIRd x\cdot\frac{u(y)\na u(x) - u(x)\na u(y)}{|x-y|}dx dy \\
&= \IIRd y\cdot\frac{u(y)\na u(x) - u(x)\na u(y)}{|x-y|}dx dy \\
&= -\frac{1}{2}\IIRd\frac{x-y}{|x-y|}(u(y)\na u(x) - u(x)\na u(y))dx dy . 
\end{align*}
Since
$$ \Div_x \frac{x-y}{|x-y|} = -\Div_y \frac{x-y}{|x-y|} = \left(\Div_z \frac{z}{|z|}\right)_{\big\vert_{z=x-y}} = \frac{2}{|x-y|} ,$$
integration by parts yields
\begin{align}\label{mom.2}
\frac{d}{dt}\IRd\frac{|x|^2}{2}u(x,t)dx = \frac{1}{2\pi}\IIRd\frac{u(x,t)u(y,t)}{|x-y|} dx dy = 2\IRd u(x,t)a(x,t)dx>0 .
\end{align}
This is one of the main differences to the classical Landau equation. The second moment increases with time and a bound is not given a-priori. We will see in Section \ref{even_sol} how to find this bound when the initial data are even.

\section{Radially symmetric solutions}\label{rad_symme}
Problem (\ref{landau}) is well understood when initial data are radially symmetric and monotonically decreasing. In \cite{GG} the authors prove the following theorem: 
\begin{theo}\label{theorem_rad}
Let $u_0$ be a nonnegative function that has finite mass, energy and entropy. Moreover let $u_0$ be  radially symmetric, monotonically decreasing and such that $u_0 \in L^p_{weak}$ for some $p>6$. Then there exists a function $u(x,t)$ smooth, positive and bounded for all time which solves 
$$
u_t = a[u] \Delta u + u^2, \quad u(x,0)=u_0.
$$
\end{theo}
We briefly highlight the ideas behind the proof of Theorem \ref{theorem_rad}. The non-local dependence on the coefficients prevents the equation to satisfy comparison principle: in fact given two functions $u_1$ and $u_2$ such that $u_1 < u_2$ for $t<t_0$ and $u_1=u_2$ at $(x_0, t_0)$ we definitely have that 
$\Delta u_1(x_0,t_0) \le  \Delta u_2(x_0,t_0)$ and $a[u_1](x_0,t_0)  \le a[u_2](x_0,t_0)$. However it is not necessarily true $a[u_1](x_0,t_0) \Delta u_1(x_0,t_0) \le a[u_2](x_0,t_0) \Delta u_2(x_0,t_0)$.  To overcome this shortcoming, the main observation in \cite{GG} is that if one proves the existence of a function $g(x)\in L^p$ for some $p>3/2$ such that $u_{0} < g$ and 
$$
a[u]\Delta g + ug <0, $$ 
then comparison principle for the linearized problem implies $u \le g $ for all $t>0$. Once higher integrability $L^p$ of $u$ is proved, standard techniques for parabolic equation such as Stampacchia's theorem yield $L^\infty$ bound for $u(x,t)$ and consequent regularity.

\section{Even initial data} \label{even_sol}
Existence of weak solutions for (\ref{landau}) with general initial data is still an open problem.  As already mentioned at the end of Section \ref{subsec_entropy}, the first obstacle that one encounters in the analysis of (\ref{landau}) is the missing bound for the second moment. This bound is essential when one seeks a-priori estimates for the gradient. In \cite{GZ} the authors overcame this problem when solutions are even. In this section we highlight the basic estimates of \cite{GZ} that will lead to construction of weak even solutions. For weak solutions we mean functions $u(x,t)$ such that
\begin{align*}
& \sqrt{u} \in L^2\left(0,T; H^1\left(\R^3,\frac{dx}{1+|x|}\right)\right),\qquad u, u\log u\in L^{\infty}(0,T; L^{1}(\R^{3})),\\
& a\in L^\infty(0,T; L^3_{loc}(\R^3)),\qquad \na a\in L^\infty(0,T; L^{3/2}_{loc}(\R^3)),
\end{align*}
that satisfy the following weak formulation 
\begin{align*}
&\int_0^T\langle \pa_{t}u\, , \phi\,\rangle dt + \int_0^T\int_{\R^3}(a\na u - u\na a)\cdot\na\phi\, dx dt = 0,
\qquad\forall\phi\in L^{\infty}(0,T; W^{1,\infty}_{c}(\R^{3})).
\end{align*}
All the computations here are formal, meaning we assume that $u$ and all related quantities have enough regularity for the mathematical manipulations to make sense. We refer to \cite{GZ} for the detailed calculations. Let 
$$
E(t) :=\IRd\frac{|x|^2}{2}u(x,t)dx ,\qquad
R(t) := 2\sqrt{\frac{E(t)}{\|u_0\|_{L^1}}},$$ 
and define $B_{R(t)}\equiv\{x\in\R^3~:~|x|<R(t)\}.$
We point out that, since
$ \int_{\R^3\backslash B_{R(t)}}u(x,t)dx\leq \frac{2 E(t)}{R(t)^2} = \frac{1}{2}\|u_0\|_{L^1}$, it follows
\begin{equation}
\int_{B_{R(t)}}u(x,t)dx = \|u_0\|_{L^1} - \int_{\R^3\backslash B_{R(t)}}u(x,t)dx \geq \frac{1}{2}\|u_0\|_{L^1} . \label{mass.R}
\end{equation}
%

\subsection*{A lower bound for $a[u]$.}
From the definition of $a[u]$ it follows
\begin{align*}
4\pi a[u](x,t) = \IRd\frac{u(y,t)}{|x-y|}dy\geq \int_{B_{R(t)}}\frac{u(y,t)}{|x-y|}dy\geq \frac{1}{R(t)+|x|}\int_{B_{R(t)}}u(y,t)dy
\geq \frac{\|u_0\|_{L^1}}{2(R(t)+|x|)}
\end{align*}
and therefore
\begin{align}\label{a.lb}
a[u](x,t)\geq \frac{1}{16\pi}\frac{\|u_0\|_{L^1}^{3/2}}{E(t)^{1/2} + |x|\|u_0\|_{L^1}^{1/2}} . 
\end{align}

\subsection*{A gradient estimate for even solutions.}
We assume here that the solution $u$ of \eqref{landau} is even w.r.t. each component of $x$, for $t\geq 0$.

%

Clearly $|x-y|\leq |x| + |y| \leq (1+|x|)(1+|y|)$ for $x,y\in\R^3$. Therefore
\begin{align*}
-4\pi\frac{d}{dt}H[u] &\geq 
\frac{1}{2}\IIRd\frac{u(x,t)u(y,t)}{(1+|x|)(1+|y|)}\left| \frac{\na u(x,t)}{u(x,t)} - \frac{\na u(y,t)}{u(y,t)} \right|^2 dx dy\\
&= \left(\IRd u(x,t)\frac{dx}{1+|x|} \right)\left(\IRd\frac{|\na u(x,t)|^2}{u(x)}\frac{dx}{1+|x|}\right) - \left| \IRd \frac{\na u(x,t)}{1+|x|}dx \right|^2 .
\end{align*}
For the assumption on $u$ it follows that 
$$ \left| \IRd \frac{\na u}{1+|x|}dx \right|^2 = \sum_{i=1}^3\left(\IRd \frac{\pa u}{\pa x_i} \frac{dx}{1+|x|}\right)^2 = 0. $$
As a consequence
\begin{align*}
-4\pi\frac{d}{dt}H[u] &\geq \left(\IRd u(x,t)\frac{dx}{1+|x|} \right)\left(\IRd\frac{|\na u(x,t)|^2}{u(x)}\frac{dx}{1+|x|}\right) .
\end{align*}
We now wish to show a positive lower bound for $\IRd u(x,t)\frac{dx}{1+|x|}$ for $0\leq t\leq T$. 
Let $R(t) = 2\sqrt{E(t)/\mass}$. It holds
$$ \IRd u(x,t)\frac{dx}{1+|x|}\geq \int_{B_{R(t)}} u(x,t)\frac{dx}{1+|x|} \geq \frac{1}{1+R(t)}\int_{B_{R(t)}}u(x,t)dx . $$
From \eqref{mass.R} it follows
\begin{align}\label{low.1}
\frac{1}{\pi}\IRd u(x,t)\frac{dx}{1+|x|}\geq \frac{1}{8\pi}\frac{\|u_0\|_{L^1}^{3/2}}{E(t)^{1/2} +\|u_0\|_{L^1}^{1/2}},\qquad t>0.
\end{align}
Since $E(t)$ is increasing, we conclude
\begin{align}\label{kappa.T}
\frac{1}{\pi}\inf_{t\in [0,T]}\IRd u(x,t)\frac{dx}{1+|x|}\geq \kappa(T) ,
\end{align}
with 
$$
\kappa(t):=\frac{1}{8\pi}\frac{\|u_0\|_{L^1}^{3/2}}{E(t)^{1/2} +\|u_0\|_{L^1}^{1/2}}.
$$
Moreover,
\begin{align}\label{dHdt.est}
\frac{d H[u]}{dt} + \kappa(t)\IRd \frac{|\na\sqrt{u(x,t)}|^2}{1+|x|}dx\leq 0,\qquad t>0.
\end{align}

\subsection*{Upper bound for $a[u]$.} It holds
\begin{align}
a[u](x,t) = \int_{|x-y|<1}\frac{u(y,t)}{|x-y|}dy + \int_{|x-y|\geq 1}\frac{u(y,t)}{|x-y|}dy \equiv I_1 + I_2 .\label{a.dec}
\end{align}
The integral $I_2$ can be estimated immediately:
$$ I_2\leq \mass . $$
For $I_1$ we first use H\"older: since $\frac{1}{|x|}$ is $L^q_{loc}(\R^3)$ for $q<3$, we get
\begin{align*}
I_1 &= \int_{|x-y|<1}\frac{u(y,t)}{|x-y|}dy \leq \left( \int_{|x-y|<1}u(y,t)^{3/2+\eps}dy \right)^{\frac{1}{3/2+\eps}}
\left( \int_{|x-y|<1}|x-y|^{-\frac{3+2\eps}{1+2\eps}}dy \right)^{\frac{1+2\eps}{3+2\eps}}\\
& \leq 4\pi \frac{1+2\eps}{4\eps}
\left( \int_{|y|<1+|x|}u(y,t)^{3/2+\eps}dy \right)^{\frac{1}{3/2+\eps}}
= \frac{(1+2\eps)\pi}{\eps}\|\sqrt{u(t)}\|_{L^{3+2\eps}(B_{1+|x|})}^2 .
\end{align*}
The interpolation inequality implies (for $0<\eps\leq 3/2$):
$$ \|\sqrt{u(t)}\|_{L^{3+2\eps}(B_{1+|x|})} \leq \|\sqrt{u(t)}\|_{L^{2}(B_{1+|x|})}^{1-\theta}\|\sqrt{u(t)}\|_{L^{6}(B_{1+|x|})}^{\theta},\quad
\theta = \frac{3}{2}\frac{1+2\eps}{3+2\eps} . $$
Then, the Sobolev embedding $H^1\hookrightarrow L^6$ implies
\begin{equation}
\|\sqrt{u(t)}\|_{L^{3+2\eps}(B_{1+|x|})} \leq C(|x|)\mass^{(1-\theta)/2}\|\sqrt{u(t)}\|_{H^1(B_{1+|x|})}^\theta . \label{est.tmp}
\end{equation}
Notice that the constant $C$ in \eqref{est.tmp} depends on $|B_{1+|x|}|$ and therefore on $|x|$. 
However, it is easy to show that such constant (assuming w.l.o.g. that it is optimal) is nonincreasing with respect to $|x|$,
thus \eqref{est.tmp} leads to
\begin{align}
\|\sqrt{u(t)}\|_{L^{3+2\eps}(B_{1+|x|})} \leq C\mass^{(1-\theta)/2}\|\sqrt{u(t)}\|_{H^1(B_{1+|x|})}^\theta .\label{est.2}
\end{align}
From \eqref{est.2} we obtain
\begin{align*}
I_1 &\leq \eps^{-1}C \|\sqrt{u(t)}\|_{H^1(B_{1+|x|})}^{2\theta }
\leq \eps^{-1}C(1 + \| \na \sqrt{u(t)}\|_{L^2(B_{1+|x|})}^{2})^{\theta} \\
&\leq \eps^{-1}C\left(1 + (2+|x|)\IRd \frac{|\na \sqrt{u(y,t)}|^2}{1+|y|} dy \right)^{\theta} \\
&\leq \eps^{-1}C(1+|x|)^\theta\left(1 + \IRd \frac{|\na \sqrt{u(y,t)}|^2}{1+|y|} dy \right)^{\theta} .
\end{align*}
The estimates of $I_1$, $I_2$ imply
$$ a[u](x,t)^{1/\theta} \leq \eps^{-1}C(1+|x|)\left(1 + \IRd \frac{|\na \sqrt{u(y,t)}|^2}{1+|y|} dy \right).$$
The entropy estimate obtained earlier  
\begin{align*}
\frac{d H[u]}{dt} + \kappa(t)\IRd \frac{|\na\sqrt{u(x,t)}|^2}{1+|x|}dx\leq 0,\qquad t>0,
\end{align*}
 leads to
$$ a[u](x,t)^{1/\theta} \leq \eps^{-1}C(1+|x|)\left(1 - \frac{1}{\kappa(t)}\frac{d H[u(t)]}{dt} \right) ,\qquad
\frac{1}{\theta} =  \frac{2(3+2\eps)}{3(1+2\eps)},\qquad 0<\eps\leq \frac{3}{2}. $$
We can restate the above estimate in a more handy way by defining $p = 1/\theta \in [1,2)$ and noticing that 
$\eps^{-1}\leq C (2-p)^{-1}$:
\begin{align}\label{est.a.1}
a[u](x,t)^{p} &\leq \frac{C}{2-p}(1+|x|)\left(1 - \frac{1}{\kappa(t)}\frac{d H[u(t)]}{dt} \right) ,\qquad 1\leq p < 2,
\end{align}
with $\kappa(t)$ given by \eqref{low.1}.

\subsection*{Lower bound for $H[u]$.}
A lower bound for $H[u(t)]$ is here showed. Being the spatial domain the whole space $\R^3$, this lower bound is not straightforward. 
To prove a lower bound for $H[u]$, we write 
$$
H[u] =\IRd u(x)\log(u(x))\chi_{\{u<1\}}\;dx + \IRd u(x)\log(u(x))\chi_{\{u>1\}}\;dx,
$$
and apply H\"older's inequality to get
\begin{align*}
-H[u] &\leq \int_{\{u<1\}}u(x)\log\frac{1}{u(x)}\, dx = \int_{\{u<1\}}u(x)^{(1-\eps)/2} u(x)^{(1+\eps)/2}\log\frac{1}{u(x)}\, dx\\
&\leq \left( \int_{\{u<1\}}u(x)^{1-\eps}dx \right)^{1/2}\left( \int_{\{u<1\}}u(x)^{1+\eps} \left( \log\frac{1}{u(x)} \right)^2 dx \right)^{1/2}.
\end{align*}
Since the function $s\in (0,1)\mapsto s^{\eps/2}\log(1/s)\in\R$ is bounded, we can estimate the term
$$
\int_{\{u<1\}}u(x)^{1+\eps} \left( \log\frac{1}{u(x)} \right)^2 dx
$$
with a constant that only depends on $\varepsilon$ and the $L^1$ norm of the initial data. 
Therefore
\begin{align}
-H[u] &\leq C_\eps\left( \int_{\{u<1\}}u(x)^{1-\eps}dx \right)^{1/2}\leq C_\eps\left( \IRd u(x)^{1-\eps}dx \right)^{1/2}.
\label{H.lb.1}
\end{align}
Let us now consider the integral
\begin{align*}
\IRd u(x)^{1-\eps}dx &= \int (1+|x|^2)^{1-\eps}u(x)^{1-\eps} (1+|x|^2)^{-(1-\eps)}  dx\\
&\leq\left( \IRd (1+|x|^2)u(x)dx \right)^{1-\eps}\left(\IRd (1+|x|^2)^{-(1-\eps)/\eps}dx\right)^{\eps}.
\end{align*}
For $\eps<2/5$ we obtain
\begin{align*}
\IRd u(x)^{1-\eps}dx &\leq C_\eps\left( \IRd (1+|x|^2)u(x)dx \right)^{1-\eps}.
\end{align*}
From the above estimate and \eqref{H.lb.1} we conclude
\begin{equation}
-H[u(t)] \leq C_\eps (1+E(t))^{(1-\eps)/2},\qquad 0<\eps<2/5,~~ t>0.\label{H.lb}
\end{equation}

\subsection*{Estimate for $E(t)$.} We recall that $E(t) = \IRd\frac{|x|^2}{2}u(x,t)dx$, $t>0$. From \eqref{mom.2}, \eqref{est.a.1} it follows
($p' \equiv p/(p-1)$):
\begin{align*}
\frac{d E(t)}{dt} &\leq 2\IRd a(x,t)u(x,t)^{1/p} u(x,t)^{1/p'}dx \leq 2\left( \IRd a(x,t)^p u(x,t) dx \right)^{1/p}\mass^{1/p'} \\
&\leq C_p \left(1 - \frac{1}{\kappa(t)}\frac{d H[u(t)]}{dt} \right)^{1/p}\left( \IRd (1+|x|) u(x,t) dx \right)^{1/p} \\
&\leq C_p \left(1 - \frac{1}{\kappa(t)}\frac{d H[u(t)]}{dt} \right)^{1/p}\left( \IRd \left(\frac{3}{2}+\frac{|x|^2}{2}\right) u(x,t) dx \right)^{1/p} \\
&\leq C_p \left(1 - \frac{1}{\kappa(t)}\frac{d H[u(t)]}{dt} \right)^{1/p}(1+E(t))^{1/p}.
\end{align*}
The definition \eqref{low.1} of $\kappa(t)$ implies that $\kappa(t)^{-1}\leq C(1+\sqrt{E(t)}) \leq C\sqrt{1+E(t)}$, so
\begin{align*}
\frac{d E(t)}{dt} &\leq C_p\left(1 - \frac{d H[u(t)]}{dt} \right)^{1/p}(1+E(t))^{\frac{3}{2p}}.
\end{align*}
Choosing $p\in (3/2,2)$, dividing the above inequality times $(1+E(t))^{3/2p}$ and integrating it in the time interval $[0,t]$ 
leads to ($E_0\equiv\IRd\frac{|x|^2}{2}u_0(x)dx$)
\begin{align*}
&(1+E(t))^{1-3/2p} - (1+E_0)^{1-3/2p} \leq C_p\int_0^t\left(1 - \frac{dH[u]}{dt}\right)^{1/p}dt'\\
&\leq C_p t^{1-1/p}\left(\int_0^t\left(1 - \frac{dH[u]}{dt}\right)dt'\right)^{1/p} = 
C_p t^{1-1/p}( t + H[u_0] - H[u(t)] )^{1/p}.
\end{align*}
By inserting \eqref{H.lb} into the above inequality we get
\begin{align*}
&(1+E(t))^{1-3/2p} - (1+E_0)^{1-3/2p} \leq C_{p,\eps} t^{1-1/p}( t + H[u_0] + (1+E(t))^{(1-\eps)/2})^{1/p} \\
&\qquad\leq C_{p,\eps}(1+t)(1+E(t))^{(1-\eps)/2p} ,\qquad \frac{3}{2}<p<2,~~ 0<\eps<\frac{2}{5},~~ t>0. 
\end{align*}
Let now $9/5 < p < 2$. We want to choose $\eps\in (0,2/5)$ such that $1-3/2p > (1-\eps)/2p$. This is equivalent to $\eps > 4-2p$. Since $p>9/5$,
it follows that $4-2p < 2/5$, so this choice of $\eps$ is admissible. Therefore Young inequality allows us to estimate the right-hand side of the above inequality
as follows
\begin{align*}
&(1+E(t))^{1-3/2p} - (1+E_0)^{1-3/2p} \leq C_{p,\eps}(1+t)^\xi + \frac{1}{2}(1+E(t))^{1-3/2p},\quad \xi = \frac{2p-3}{2p-4+\eps},
\end{align*}
and so we conclude
\begin{equation}
E(t) \leq C_{p,\eps}(1 + t^{2p/(2p-4+\eps)})\qquad t>0,\quad \frac{9}{5}<p<2, \quad 4-2p<\eps<\frac{2}{5}.\label{E.ub}
\end{equation}
For example, if $p = (9/5+2)/2 = 19/10$ and $\eps = (4-2p+2/5)/2 = 3/10$, then $2p/(2p-4+\eps) = 38$.

Bound \eqref{E.ub} means that $E\in L^\infty_{loc}(0,\infty)$. A few consequences of this fact are, for example, that for any $T>0$:
\begin{enumerate}
\item the quantity $\kappa(t)$ defined in \eqref{low.1} and appearing e.g. in \eqref{est.a.1} is uniformely positive for $t\in [0,T]$;
\item the entropy $H[u(t)]$ has a uniform lower bound for $t\in [0,T]$;
\item ineq.~\eqref{dHdt.est} and the mass conservation yield the following estimate:
\begin{equation}
\|\sqrt{u}\|_{L^2(0,T; H^1(\R^3,\gamma(x)dx)}\leq C_T,\qquad \gamma(x)\equiv (1+|x|)^{-1};\label{est.nau}
\end{equation}
\item the lower bound \eqref{a.lb} for $a$ is uniform in $t\in[0,T]$.
\end{enumerate}

\section{Conditional smoothness}

\subsection{Conditional regularity estimates.} This section concerns results of conditional regularity of solutions to (\ref{landau}). These results are based upon a so-called {\em $\varepsilon$-Poincar\'e inequality}. 
We say that $u$ satisfies the $\varepsilon$-Poincar\'e inequality if given $\varepsilon>0$ as small as one wishes, 
there exists a constant $C_\varepsilon$ such that the following inequality holds true
 \begin{align}\label{eqn:epsilon_Poincare_inequality strongerII}
  \begin{array}{l}
	  \int_{\mathbb{R}^d}  u \phi^2\;dx \leq \varepsilon \int_{\mathbb{R}^d} a[u] |\nabla\phi|^2 \;dx + C_\varepsilon\int_{\mathbb{R}^d}	\phi^2\;dx,
  \end{array}	  
\end{align} 
for any $\phi\in L^1_{loc}(\R^3)$ that makes the right-hand side of \eqref{eqn:epsilon_Poincare_inequality strongerII} convergent.

\begin{theo}[Conditional regularity]\label{thr.reg} Let $u$ be a solution to \eqref{landau}. Assume $u$ is such that (\ref{eqn:epsilon_Poincare_inequality strongerII}) holds true. 
Then for any $s_1>1$, $s_2>\frac{1}{3}$, $T>0$, $R>0$ there exist constants $C_1=C_1(T,u_0,s_1,R)$, $C_2=C_2(T,u_0,s_2)$ such that 
\begin{align*}
\|u\|_{L^\infty(B_R\times(t,T))} &\le C(T,u_0, s_1,R)\left(\frac{1}{t}+ 1\right)^{s_1},\qquad t\in (0,T),\\
\|a[u]\|_{L^\infty(\R^3\times (t,T))} &\le C(T,u_0, s_2)\left(\frac{1}{t}+ 1\right)^{s_2},\qquad t\in (0,T),
\end{align*}
where $B_R\subset\R^3$ is any ball of radius $R$.
\end{theo}

Weighted Sobolev and Poincare's inequalities have been used to obtain informations about eigenvalues for the Schr\"odinger and degenerate elliptic operators  \cite{CWW, CW1, CW2, FP78, SW}.  
Inspired by the similarity of (\ref{landau}) with the degenerate operator $L = -\textrm{div}(a[u] \nabla )-u$, in \cite{GG17} the new inequality (\ref{eqn:epsilon_Poincare_inequality strongerII}) has been proposed. We refer to \cite{GG17} for discussions about (\ref{eqn:epsilon_Poincare_inequality strongerII}). While (\ref{eqn:epsilon_Poincare_inequality strongerII}) is always true provided $u$ solves the Landau equation for soft-potentials \cite{GG17}, the validity of (\ref{eqn:epsilon_Poincare_inequality strongerII}) for Coulomb interactions is still an open question, undoubtedly a very interesting and fundamental one. Consequently the results in Theorem \ref{thr.reg} should be viewed as conditional.

Very interesting is the rate of decay in the estimate for $\|u\|_{L^\infty(B_R\times(t,T))}$. In fact one would expect a decay with a rate similar to the heat kernel $1/t^{3/2}$. However thanks to a combination of (\ref{eqn:epsilon_Poincare_inequality strongerII}) and a non-local Poincare's inequality proven in \cite{GKS} we obtain a decay that can be made arbitrary close to $1/t$. 
 
The proof of Theorem \ref{thr.reg} is divided into several Lemmas and Propositions. We will make use of the following

\begin{lemma}[Weighted Sobolev inequality]\label{lem:Inequalities Sobolev weight aII}
Let $u$ be a solution to  (\ref{landau}). Any smooth function $\phi$ satisfies
    \begin{align*}
     \left ( \int_{I}\int \phi^q a[u]\;dxdt \right )^{2/q}& \leq C\left ( \int_{I} \int a[u] |\nabla \phi|^2 \;dxdt +  \sup \limits_{I} \int  \phi^{2}\;dx \right ),
    \end{align*}
    with
    \begin{align*}
    q \in \left( 1 , 2\left(1+\frac{2}{3}\right)\right). 
    \end{align*}
  \end{lemma}
  \begin{proof}
  We refer to  \cite{GG17} for a detailed proof.
  \end{proof}

We define $u_k := (u-k)_+$ for a generic constant $k>0$.

   \begin{prop}\label{prop:energy_identity} The following inequality holds:
    \begin{align}\label{est3-1}
      & \partial_t\int \eta^2 u_k^p\;dx + \frac{4(p-1)}{p}\int a|\nabla (\eta u_k^{p/2})|^2\;dx {+\frac{p(p-1)\tau}{2}\int \frac{u_k^{p-2}}{u^3}|\na u_k|^4 \eta^2 dx}\\
      &\qquad\leq \textnormal{(I)} + \textnormal{(II)} 
      {+ C\tau\int \eta^2 u_k^p\; dx + C(p)\tau\int \left(1+\frac{|\na\eta|^{4p}}{\eta^{4p}}\right)\eta^2 dx,}\nonumber
    \end{align}
    where 
    \begin{align*}		  
      \textnormal{(I)} & := \frac{4(p-2)}{p}\int u_k^{p/2}(a\nabla (\eta u_k^{p/2}),\nabla \eta)\;dx+\frac{4}{p}\int u_k^{p}(a\nabla \eta,\nabla \eta)\;dx,\\
  	\textnormal{(II)} & := \int u_k^p (\nabla a,\nabla (\eta^2))\;dx+(p-1)\int u\eta^2 u_k^p\;dx+pk\int u\eta^2 u_k^{p-1} \;dx.
    \end{align*}   
  \end{prop}
  \begin{proof}
    Consider 
    $$
\psi= p  \; \eta^2\; u^{p-1}_k
 $$
 as test function for (\ref{landau}). A direct computation yields, 
    \begin{align*}
&p \int \eta^2 u^{p-1}_k \partial_t u_k\;dx \\
        & = - p\int (a\nabla u,\nabla (\eta^2 u_k^{p-1}))\;dx
        +p \int (u\nabla a,\nabla (\eta^2 u_k^{p-1}))\;dx\\
         &= \widetilde{\textnormal{(I)}} + \textnormal{(II)}.
    \end{align*}
    Expanding the first integral, we have the expression:
    \begin{align*}
    \int (a\nabla u,\nabla (\eta^2 u_k^{p-1}))\;dx =   \int (p-1) \eta^2 u_k^{p-2}(a\nabla u_k,\nabla u_k)+2 u_k^{p-1}\eta(a\nabla u_k,\nabla \eta)\;dx.
    \end{align*}	  
    Let us rewrite this expression in a more convenient form. Note the elementary identity
    \begin{align*}
      (a\nabla (\eta u_k^{p/2}),\nabla (\eta u_k^{p/2}))  
      & = \frac{p^2}{4} u_k^{p-2}\eta^2 (a\nabla u_k,\nabla u_k)+p\eta u_k^{p-1}(a\nabla u_k,\nabla \eta)+u_k^{p}(a\nabla \eta,\nabla \eta),
    \end{align*}
    and use it to write,
    \begin{align*}		
      & (p-1)\eta^2 u_k^{p-2}(a\nabla u_k,\nabla u_k)+2u_k^{p-1}\eta (a\nabla u_k,\nabla \eta) \\
      & = \frac{4(p-1)}{p^2}(a\nabla (\eta u_k^{p/2}),\nabla (\eta u_k^{p/2}))\\
      & \;\;\;\; - \frac{(2p-4)}{p} u_k^{p-1}\eta(a\nabla u_k,\nabla \eta)-\frac{4(p-1)}{p^2} u_k^p(a\nabla \eta,\nabla \eta).	  	
    \end{align*}
    Further, another elementary identity says
    \begin{align*}	
      u_k^{p-1}\eta (a\nabla u_k,\nabla \eta)  
	  = \frac{2}{p} u_k^{p/2} (a  \nabla (\eta u_k^{p/2}),\nabla \eta )-\frac{2}{p}u_k^p (a\nabla \eta,\nabla \eta). 		
    \end{align*}		
    Combining the above, it follows that
    \begin{align*}	
      & (p-1)\eta^2 u_k^{p-2}(a\nabla u_k,\nabla u_k)+2u_k^{p-1}\eta (a\nabla u_k,\nabla \eta)\\
      & = \frac{4(p-1)}{p^2}(a\nabla (\eta u_k^{p/2}),\nabla (\eta u_k^{p/2}))\\	
      & \;\;\;\;-\frac{4(p-2)}{p^2}u_k^{p/2}(a\nabla (\eta u_k^{p/2}),\nabla \eta)-\frac{4}{p^2}u_k^p(a\nabla \eta,\nabla \eta).
    \end{align*}	
    In particular,
    \begin{align*}	
   \widetilde{\textnormal{(I)}}=    & -\frac{4(p-1)}{p}\int (a\nabla (\eta u_k^{p/2}),\nabla (\eta u_k^{p/2}))\;dx\\
      & +\frac{4(p-2)}{p}\int u_k^{p/2}(a\nabla (\eta u_k^{p/2}),\nabla \eta)\;dx+\frac{4}{p}\int u_k^{p}(a\nabla \eta,\nabla \eta)\;dx.
    \end{align*}	
    Thus,
    \begin{align*}	
      & \frac{d}{dt}\int \eta^2 u_k^p\;dx + \frac{4(p-1)}{p}\int (a\nabla (\eta u_k^{p/2}),\nabla (\eta u_k^{p/2}))\;dx \\
      & = \frac{4(p-2)}{p}\int u_k^{p/2}(a\nabla (\eta u_k^{p/2}),\nabla \eta)\;dx+\frac{4}{p}\int u_k^{p}(a\nabla \eta,\nabla \eta)\;dx\\
      & \;\;\;\;+ p\int (u\nabla a,\nabla (\eta^2u_k^{p-1} ) )\;dx.	   	  	
    \end{align*}	
    We now analyze ${\textnormal{(II)}}$. Since
    \begin{align*}
      (\nabla a,u\nabla (\eta^2 u_k^{p-1})) & =  uu_k^{p-1}(\nabla a,\nabla (\eta^2))+ (p-1)uu_k^{p-2}\eta^2(\nabla a,\nabla u_k)\\
      & = uu_k^{p-1}(\nabla a,\nabla (\eta^2))+(p-1)( u_k^{p-1}+ku_k^{p-2})\eta^2 (\nabla a,\nabla u_k)\\
      & = (u_k^p+ku_k^{p-1})(\nabla a,\nabla (\eta^2))\\
      & \;\;\;\;  +\eta^2 (\nabla a,\nabla ( \frac{p-1}{p}u_k^p+ku_k^{p-1}) ),	    
    \end{align*}
    it follows that
    \begin{align*}
      \textnormal{(II)}  & = p\int (u_k^p+ku_k^{p-1} )(\nabla a,\nabla (\eta^2))\;dx \\
        & \;\;\;\;-p\int \left ( \frac{p-1}{p}u_k^{p}+ku_k^{p-1}\right )\dive(\eta^2\nabla a)\;dx.
    \end{align*}		
    {From the above inequality and the Poisson equation it follows}
    \begin{align*}			  
  	\textnormal{(II)} & = p\int (u_k^p+ku_k^{p-1}) (\nabla a,\nabla (\eta^2))\;dx-\int ((p-1)u_k^p+pku_k^{p-1}) (\nabla a,\nabla (\eta^2))\;dx\\
        & \;\;\;\;+\int u \eta^2 ((p-1)u_k^p+pku_k^{p-1}) \;dx\\
        & = \int u_k^p (\nabla a,\nabla (\eta^2))\;dx+\int u\eta^2 \left ( (p-1)u_k^p+pku_k^{p-1} \right )\;dx.
    \end{align*}

This finishes the proof of the Lemma.
  \end{proof}	

\begin{lemma} \label{lemma:2}
Let $p>1$, then we have the inequality
    \begin{align*}
      & \frac{d}{dt}\int \eta^2 u_k^p\;dx + \frac{(p-1)}{p}\int a|\nabla (\eta u_k^{p/2})|^2\;dx\\
      & \leq {{(p-1)\int \eta^2 uu_k^p\;dx+p k \int \eta^2 u u_k^{p-1}\;dx}} \\
      & \;\;\;\;+C(p) \int u_k^p (a\nabla \eta,\nabla \eta) \;dx-\int u_k^p \eta \textrm{Tr}\;(aD^2\eta))\;dx ,
         \end{align*}
    where $C(p)$ denotes a constant that is bounded when $p>1$.
  \end{lemma}

  \begin{proof}	
    We proceed to bound from above the first term $\textnormal{(I)}$ and the first term of $\textnormal{(II)}$ resulting from Proposition \ref{prop:energy_identity}.  The aim is to estimate these terms as 
    \begin{align*}
     \frac{4(p-2)}{p}\int u_k^{p/2}(a\nabla (\eta u_k^{p/2},\nabla \eta)\;dx& +\int u_k^p (\nabla a,\nabla (\eta^2))\;dx  \\
     \le & c_1 \int (a\nabla (\eta u_k^{p/2}),\nabla (\eta u_k^{p/2}))\;dx + \textrm{\em lower order terms,}
\end{align*}
    where $c_1 < \frac{4(p-1)}{p}$.
   For the first term we use Cauchy-Schwarz inequality 
    \begin{align}
      & \left | \frac{4(p-2)}{p}(a\nabla (\eta u_k^{p/2}),u_k^{p/2}\nabla \eta) \right | \nonumber \\
      & \leq \frac{2(p-1)}{p}(a\nabla (\eta u_k^{p/2}),\nabla (\eta u_k^{p/2}))+ \frac{2(p-2)^2}{p(p-1)}u_k^p(a\nabla \eta,\nabla \eta). \label{est2-1}
    \end{align}			
    For the first term in $\textnormal{(II)}$ we use the identity
    \begin{align*}
      \dive(a u_k^p \nabla (\eta^2)) = a \dive(u_k^p\nabla (\eta^2))+u_k^p(\nabla a,\nabla (\eta^2)),
    \end{align*} 
    and conclude that
    \begin{align*}
      \int u_k^p (\nabla a,\nabla (\eta^{2}))\;dx & = -\int a \dive(u_k^p\nabla (\eta^2))\;dx\\
       & = -\int a u_k^p \Delta (\eta^2)\;dx- \int (a\nabla u_k^p,\nabla \eta^2)\;dx.
    \end{align*}
    Since
    \begin{align*}
       \eta \nabla u_k^{p/2} = \nabla (\eta u_k^{p/2})- u_k^{p/2}\nabla \eta,	
    \end{align*}
    {Young's inequality} yields
    \begin{align*}
       &{-}\int (a \nabla u_k^p,\nabla \eta^2)\;dx ~ {= -4\int u_k^{p/2}(a\eta\na u_k^{p/2}, \nabla\eta)}\\
       &\qquad = {-}4\int u_k^{p/2}(a \nabla (\eta u_k^{p/2}),\nabla \eta)\; dx +4\int u_k^p (a\nabla \eta,\nabla \eta)\;dx \\
       &\qquad \leq 2\varepsilon\int (a\nabla (\eta u_k^{p/2}),\nabla (\eta u_k^{p/2}) )\;dx + \left(\frac{2}{\varepsilon}+4\right)\int u_k^{p}(a\nabla \eta,\nabla\eta)\;dx .
    \end{align*}
    Thus
    \begin{align}\label{est1-1}
      \int u_k^p (\nabla a,\nabla (\eta^{2}))\;dx  \le &-\int u_k^p \textrm{Tr} (a D^2(\eta^2))\;dx +  2\varepsilon\int (a\nabla (\eta u_k^{p/2}),\nabla (\eta u_k^{p/2}) )\;dx \nonumber\\
    &  + \left(\frac{2}{\varepsilon}+4\right)\int u_k^{p}(a\nabla \eta,\nabla\eta)\;dx.
      \end{align}
  Substituting (\ref{est1-1}) and  (\ref{est2-1}) into  (\ref{est3-1}) we get by choosing $\varepsilon < \frac{p-1}{2p}$
   \begin{align*}
       \frac{d}{dt}\int \eta^2 u_k^p\;dx &+ \frac{(p-1)}{p} \int (a\nabla (\eta u_k^{p/2}),\nabla (\eta u_k^{p/2}))\;dx \\
      \le &C(p)  \int u_k^{p}(a\nabla \eta,\nabla \eta)\;dx+(p-1)\int \eta^2 u u_k^p\;dx\\
      &+pk\int  \eta^2  u u_k^{p-1} \;dx -\int u_k^p \textrm{Tr} (a D^2(\eta^2))\;dx. 
    \end{align*}   
  \end{proof}

\begin{lemma}\label{lemma3}
 We have
\begin{align*}
(p-1)\int_t^T \int \eta^2 uu_k^p\;dxds  \le & \; \varepsilon (p-1)\int_t^T \int_{Q_R} a |\nabla(\eta u_k^{p/2})|^2 \;dxds + C(R,\varepsilon,p)\int_t^T\int_{Q_R}	\eta^2 u_k^{p}\;dxds,
\\
 p k \int_t^T\int \eta^2 u u_k^{p-1}\;dxds  \le & \; p \varepsilon \int_t^T  \int a |\nabla(\eta u_k^{p/2})|^2 \;dxds + C(R,\varepsilon,p)\int_t^T\int 	\eta^2 u_k^{p}\;dxds\\
& +\; \;2p k^2  \int_0^T  \int \eta^2 u_k^{p-1}\;dxds. 
\end{align*}
\end{lemma}

\begin{proof}
We use here the $\varepsilon$-Poincare's inequality (\ref{eqn:epsilon_Poincare_inequality strongerII}) with 
$$
\phi = \eta u_k^{p/2}
$$
and get 
 \begin{align*}
  \begin{array}{l}
	 \int \eta^2 uu_k^p\;dx \leq \varepsilon \int a |\nabla(\eta u_k^{p/2})|^2 \;dx + C(R,\varepsilon)\int 	\eta^2 u_k^{p}\;dx.
  \end{array}	  
\end{align*}

For the second inequality we get 
\begin{align*}
 p k \int_t^T\int \eta^2 u u_k^{p-1}\;dxds= &\;  p k \int_t^T\int \eta^2 [u\chi_{\{u_k\ge k\}} +  u\chi_{\{u_k\le k\}}]u_k^{p-1}\;dxds\\
= & \; p k \int_t^T\int \eta^2 u\chi_{\{u_k\ge k\}} u_k^{p-1}\;dxds +  p k \int_t^T\int \eta^2\underbrace{u\chi_{\{u_k\le k\}}}_{u\le 2k}u_k^{p-1}\;dxds\\
 \le &\; p\int_t^T \int  \eta^2 u u_k^{p}\;dxds + \;2p k^2  \int_0^T  \int \eta^2 u_k^{p-1}\;dxds \\
\le &\; p \varepsilon \int_t^T  \int a |\nabla(\eta u_k^{p/2})|^2 \;dxds + C(R,\varepsilon,p)\int_t^T\int 	\eta^2 u_k^{p}\;dxds\\
& +\; \;2p k^2  \int_0^T  \int \eta^2 u_k^{p-1}\;dxds 
 \end{align*}
 using  (\ref{eqn:epsilon_Poincare_inequality strongerII}) once more. 

\end{proof}

\begin{cor}  \label{cor_iterat}
  Fix times $0<T_1<T_2<T_3<T$, $p>1$ and a cut-off function $\eta(v)$. Then, we have the following inequality
  \begin{align*}
   \sup \limits_{T_2 \leq t\leq T_3} \left \{ \int (\eta u_k^{p/2})^2\;dx\right\} &+ \frac{(p-1)}{4p}  \int_{T_2}^{T_3}\int a|\nabla (\eta u_k^{p/2})|^2\;dxdt\\
      \le  &\; \left(\frac{1}{T_2-T_1} + C(p,\varepsilon, R)\right)\int_{T_1}^{T_3} \int \eta^2u_k^p\;dxdt \\
      &+2 p k^2 \int_{T_1}^{T_3}\int \eta^2 u_k^{p-1}\;dxdt\\
            &+C(p) \int_{T_1}^{T_3}\int  u_k^p (a\nabla \eta,\nabla \eta) \;dxdt+ \int_{T_1}^{T_3}\int a u_k^p \eta |\Delta \eta|\;dxdt. 	
    \end{align*}	

    \end{cor}  
  
\begin{proof}  
  We start with the bound found in Lemma \ref{lemma:2} 
  \begin{align*}
      & \frac{d}{dt}\int \eta^2 u_k^p\;dx + \frac{(p-1)}{p}\int a|\nabla (\eta u_k^{p/2})|^2\;dx\\
      & \leq {{(p-1)\int \eta^2 uu_k^p\;dx+p k \int \eta^2 u u_k^{p-1}\;dx}} \\
      & \;\;\;\;+C(p) \int u_k^p (a\nabla \eta,\nabla \eta) \;dx-\int a u_k^p \eta \Delta\eta\;dx .
         \end{align*}
         
         

    Integrating this inequality from $t_1$ to $t_2$ shows that the term  
    \begin{align*}
      \int \eta^2 u_k^p(t_2)\;dx-\int \eta^2u_k^p(t_1)\;dx + \frac{(p-1)}{p}\int_{t_1}^{t_2}\int a|\nabla (\eta u_k^{p/2})|^2\;dxdt
    \end{align*}	
    is bounded by 
    \begin{align*}	
     (p-1) \int_{t_1}^{t_2}\int \eta^2 uu_k^p\;dxdt+p k  \int_{t_1}^{t_2}\int \eta^2 u u_k^{p-1}\;dxdt\\
      +C(p) \int_{t_1}^{t_2}\int  u_k^p (a\nabla \eta,\nabla \eta) \;dxdt- \int_{t_1}^{t_2}\int a u_k^p \eta \Delta \eta\;dxdt. 	
    \end{align*}	
    
For a fixed $t_2 \in (T_2,T_3)$, we take the average with respect to $t_1\in (T_1,T_2)$ in both sides of the inequality. This yields 
  \begin{align*}
  \frac{1}{T_2-T_1}\int_{T_1}^{T_2} \int \eta^2 u_k^p(t_2)\;dxdt_1 &+ \frac{(p-1)}{p}\frac{1}{T_2-T_1}\int_{T_1}^{T_2}  \int_{t_1}^{t_2}\int a|\nabla (\eta u_k^{p/2})|^2\;dxdtdt_1\\
      \le  &\; \frac{1}{T_2-T_1}\int_{T_1}^{T_2} \int \eta^2u_k^p(t_1)\;dxdt_1\\
       & + (p-1) \frac{1}{T_2-T_1}\int_{T_1}^{T_2} \int_{t_1}^{t_2}\int \eta^2 uu_k^p\;dxdtdt_1\\
      &+p k  \frac{1}{T_2-T_1}\int_{T_1}^{T_2}\int_{t_1}^{t_2}\int \eta^2 u u_k^{p-1}\;dxdtdt_1  \\
      &+C(p) \frac{1}{T_2-T_1}\int_{T_1}^{T_2}\int_{t_1}^{t_2}\int  u_k^p (a\nabla \eta,\nabla \eta) \;dxdtdt_1\\
     & - \frac{1}{T_2-T_1}\int_{T_1}^{T_2} \int_{t_1}^{t_2}\int a u_k^p \eta \Delta \eta\;dxdtdt_1,
    \end{align*}	 
which implies 
\begin{align*}
  \int \eta^2 u_k^p(t_2)\;dx &+ \frac{(p-1)}{p}  \int_{T_2}^{t_2}\int a|\nabla (\eta u_k^{p/2})|^2\;dxdt\\
      \le  &\; \frac{1}{T_2-T_1}\int_{T_1}^{T_2} \int \eta^2u_k^p(t)\;dxdt \\
      &+ (p-1) \int_{T_1}^{t_2}\int \eta^2 uu_k^p\;dxdt+p k \int_{T_1}^{t_2}\int \eta^2 u u_k^{p-1}\;dxdt  \\
      &+C(p) \int_{T_1}^{t_2}\int  u_k^p (a\nabla \eta,\nabla \eta) \;dxdt + \int_{T_1}^{t_2}\int a u_k^p \eta |\Delta \eta|\;dxdt. 	
    \end{align*}	
   Since this holds for every $t_2\in (T_2,T_3)$, this implies the inequality
  \begin{align*}
   \sup \limits_{T_2 \leq t\leq T_3} \left \{ \int \eta^2 u_k^p(t)\;dx\right\} &+ \frac{(p-1)}{p}  \int_{T_2}^{T_3}\int a|\nabla (\eta u_k^{p/2})|^2\;dxdt\\
      \le  &\; \frac{1}{T_2-T_1}\int_{T_1}^{T_3} \int \eta^2u_k^p(t)\;dxdt \\
      &+ (p-1) \int_{T_1}^{T_3}\int \eta^2 uu_k^p\;dxdt+p k \int_{T_1}^{T_3}\int \eta^2 u u_k^{p-1}\;dxdt\\
            &+C(p) \int_{T_1}^{T_3}\int  u_k^p (a\nabla \eta,\nabla \eta) \;dxdt + \int_{T_1}^{T_3}\int a u_k^p \eta |\Delta \eta|\;dxdt. 	
    \end{align*}	
  As the last step we use Lemma \ref{lemma3} with $\varepsilon < \frac{p-1}{4p^2}$ and get 
   \begin{align*}
   \sup \limits_{T_2 \leq t\leq T_3} \left \{ \int \eta^2 u_k^p(t)\;dx\right\} &+ \frac{(p-1)}{4p}  \int_{T_2}^{T_3}\int a|\nabla (\eta u_k^{p/2})|^2\;dxdt\\
      \le  &\; \frac{1}{T_2-T_1}\int_{T_1}^{T_3} \int \eta^2u_k^p(t)\;dxdt \\
      &+ C(p,\varepsilon, R) \int_{T_1}^{T_3}\int \eta^2 u_k^p\;dxdt+2 p k^2 \int_{T_1}^{T_3}\int \eta^2 u_k^{p-1}\;dxdt\\
            &+C(p) \int_{T_1}^{T_3}\int  u_k^p (a\nabla \eta,\nabla \eta) \;dxdt+ \int_{T_1}^{T_3}\int a u_k^p \eta |\Delta \eta|\;dxdt. 	
    \end{align*}

 \end{proof}

 \begin{cor}\label{cor_final_lp}
 We have 
 \begin{align*}
   \sup \limits_{T_2 \leq t\leq T_3} \left \{ \int  u^p(t)\;dx\right\} &+ \frac{(p-1)}{4p}  \int_{T_2}^{T_3}\int a|\nabla ( u^{p/2})|^2\;dxdt\\
      \le  &\; \left(\frac{1}{T_2-T_1}+ C(p,\varepsilon)\right)\int_{T_1}^{T_3} \int u^p(t)\;dxdt. 	
    \end{align*}	
 
 \end{cor}
 \begin{proof}
 It is a consequence of Corollary \ref{cor_iterat} if $\eta =1$ and $k=0$. 
 \end{proof}
 
 \begin{lemma}[Gain in integrability]\label{gain_n}
 For each $p>1$ and integer $n\ge 0$ we have 
  \begin{align*}
   \sup \limits_{T/4 \leq t\leq T} \left \{ \int  u^{p+n}(t)\;dx\right\}  \le C(p,n) \left(\frac{1}{T}+1\right)^{n+1} \int_{0}^{T} \int u^p(t)\;dxdt. 
 \end{align*}
 \end{lemma}
 
 \begin{proof}
The proof is based on iterating Corollary \ref{cor_final_lp} with a non-local weighted Poincare's inequality proven in \cite{GKS}: for each $p>0$  any smooth function $u\ge 0$ satisfies 
\begin{align}\label{eq:GressmannKriegerStrain}
  \int_{\mathbb{R}^d} u^{p+1}\;dx \leq \left (\frac{p+1}{p} \right )^2\int_{\mathbb{R}^d} a[u] |\nabla (u^{{p}/{2}})|^2 \;dx. 
\end{align}

Consider a sequence of times 
\begin{align*}
T_{n} = \frac{T}{4}\left ( 1- \frac{1}{2^{n-1}} \right ). 
  \end{align*}	

We start with Corollary \ref{cor_final_lp} which states that for each $p>1$
\begin{align*}
   \sup \limits_{T_2 \leq t\leq T} \left \{ \int  u^p(t)\;dx\right\} &+ \frac{(p-1)}{4p}  \int_{T_2}^{T}\int a[u]|\nabla ( u^{p/2})|^2\;dxdt\\
      \le  &\; \left(\frac{1}{T_2}+ C(p,\varepsilon)\right)\int_{0}^{T} \int u^p(t)\;dxdt. 	
    \end{align*}	
    Inequality (\ref{eq:GressmannKriegerStrain}) implies 
    \begin{align*}
  \frac{p(p-1)}{4(p+1)^2}  \int_{T_2}^{T}\int  u^{p+1}\;dxdt  \le  &\; \left(\frac{1}{T_2}+ C(p,\varepsilon)\right)\int_{0}^{T} \int u^p(t)\;dxdt. 
    \end{align*}
We now apply the energy inequality to $u^{p+1}$ 
\begin{align*}
   \sup \limits_{T_3 \leq t\leq T} \left \{ \int  u^{p+1}(t)\;dx\right\} &+ \frac{p}{4(p+1)}  \int_{T_3}^{T}\int a[u]|\nabla ( u^{(p+1)/2})|^2\;dxdt\\
      \le  &\; \left(\frac{1}{T_3-T_2}+ C(p,\varepsilon)\right)\int_{T_2}^{T} \int u^{p+1}(t)\;dxdt\\
      \le  &\;  \frac{4(p+1)^2}{p(p-1)} \left(\frac{1}{T_3-T_2}+ C(p,\varepsilon)\right) \left(\frac{1}{T_2}+ C(p,\varepsilon)\right)\int_{0}^{T} \int u^p(t)\;dxdt \\
      \le \:& 2^6\frac{(p+1)^2}{p(p-1)}\left(\frac{1}{T}+ C(p,\varepsilon)\right)^2\int_{0}^{T} \int u^p(t)\;dxdt.
    \end{align*}	
    Iterating the process we get 
    \begin{align*}
   \sup \limits_{T_{n+2} \leq t\leq T} \left \{ \int  u^{p+n}(t)\;dx\right\}   \le  2^{\sum_1^{n+2} k}C(p)^n\left(\frac{1}{T}+ 1\right)^{n+1}\int_{0}^{T} \int u^p(t)\;dxdt.
    \end{align*}	
Since $T_n\le T/4$ for any $n\ge 0$ we conclude 
\begin{align*}
   \sup \limits_{T/4 \leq t\leq T} \left \{ \int  u^{p+n}(t)\;dx\right\}   \le  2^{n(n+1)}C(p)^n\left(\frac{1}{T}+ 1\right)^{n+1}\int_{0}^{T} \int u^p(t)\;dxdt,
    \end{align*}	  
and the lemma is proven.

 \end{proof}
 
 \subsection{Global $L^pL^p$ estimates}
  \begin{lemma}\label{L1L3}
There exists a constant that only depends on $T$ and the initial data $u_0$ such that 
$$
\|u\|_{L^1(0,T;L^3(\mathbb{R}^3, \gamma^3dx))} \le C(T,u_0).
$$
 \end{lemma}
 
 \begin{proof}
 We start with the classical Sobolev inequality in three dimensions: 
 $$
 \left(\int_{\mathbb{R}^3} g^6 \;dx \right)^\frac{1}{3} \le C \int_{\mathbb{R}^3} |\nabla g|^2 \;dx,
 $$
and apply it to 
 $
 g = \frac{\sqrt{u}}{(1+|x|)^{1/2}}.
 $
 Since 
 $$
 |\nabla g | \le \frac{ |\nabla \sqrt{u} |}{(1+|x|)^{1/2}} +  \sqrt{u},
 $$
 Sobolev inequality yields
  $$
 \left(\int_{\mathbb{R}^3} \frac{u^3}{(1+|x|)^{3}}  \;dx \right)^{\frac{1}{3}} \le C \int_{\mathbb{R}^3}\frac{ |\nabla \sqrt{u} |^2}{(1+|x|)} +  u \;dx.
 $$
 Integrating both sides in the time interval $(0,T)$ we get 
 \begin{align}\label{est_L1L3}
\int_0^T \left(\int_{\mathbb{R}^3} \frac{u^3}{(1+|x|)^{3}}  \;dx \right)^{\frac{1}{3}}dt & \le C\int_0^T \int_{\mathbb{R}^3}\frac{ |\nabla \sqrt{u} |^2}{(1+|x|)} \;dxdt +   \int_0^T \int_{\mathbb{R}^3}u \;dxdt  \nonumber \\
& \le C(T,u_0)
 \end{align}
 using mass conservation and estimate (\ref{est.nau}). 
\end{proof}

 \begin{lemma}\label{LpLp}
There exists a constant that only depends on $T$ and the initial data $u_0$ such that 
$$
\|u\|_{L^{5/3}(0,T;L^{5/3}(\mathbb{R}^3))} \le C(T,u_0).
$$

 \end{lemma}

 \begin{proof}
Interpolation yields
\begin{align*}
\int_{\mathbb{R}^3}u^p \;dx &= \int_{\mathbb{R}^3}u^{p\theta} u^{p(1-\theta)}(1+|x|)^m(1+|x|)^{-m}  \;dx \\
&\le  \left(\int_{\mathbb{R}^3}u^{pp_1\theta} (1+|x|)^{p_1m}\;dx\right)^{\frac{1}{p_1}}  \left(\int_{\mathbb{R}^3} u^{p(1-\theta)p_2}(1+|x|)^{-mp_2}  \;dx \right)^{\frac{1}{p_2}},
 \end{align*}
 with $\frac{1}{p_1} + \frac{1}{p_2} =1$ and $\theta <1$. For $m=1$, $p_1 = 3/2$, $p_2 = 3$, $p=5/3$ and $\theta = 2/5$ we get 
  \begin{align*}
\int_{\mathbb{R}^3}u^p \;dx \le  \left(\int_{\mathbb{R}^3}u (1+|x|)^{3/2}\;dx\right)^{\frac{3}{5}} \left(\int_{\mathbb{R}^3} u^{3}(1+|x|)^{-3}  \;dx \right)^{\frac{1}{3}}\\
\le \left(\int_{\mathbb{R}^3}u (1+|x|)^{2}\;dx\right)^{\frac{3}{5}} \left(\int_{\mathbb{R}^3} u^{3}(1+|x|)^{-3}  \;dx \right)^{\frac{1}{3}}.
  \end{align*}
  Integrating in the time interval $(0,T)$ we get 
  \begin{align*}
\int_0^T\int_{\mathbb{R}^3}u^p \;dxdt &\le \int_0^T\left(\int_{\mathbb{R}^3}u (1+|x|)^{2}\;dx\right)^{\frac{3}{5}} \left(\int_{\mathbb{R}^3} u^{3}(1+|x|)^{-3}  \;dx \right)^{\frac{1}{3}} \;dt \\
&\le C(T,u_0) \int_0^T  \left(\int_{\mathbb{R}^3} u^{3}(1+|x|)^{-3}  \;dx \right)^{\frac{1}{3}} \;dt  \le C(T,u_0),
\end{align*}
  using conservation of mass and bound of the second momentum for the second inequality and (\ref{est_L1L3}) in the last inequality. 
 
 \end{proof}

 \subsection{Gain in integrability} 
The aim of this section is to show that $f$ has enough integrability for $a[u]$ to be uniformly bounded in space and time. A consequence of interpolation and H\"older's inequality is that $a[u](x,t)$, defined as 
$$
a[u](x,t) := \frac{1}{4\pi}\int_{\mathbb{R}^3} \frac{u(y)}{|x-y|}\;dy,
$$
is uniformly bounded in space and time if $u$ belongs to $L^\infty(L^p(\mathbb{\R}^3))$ with $p>\frac{3}{2}$. This is what we will show next, combining inequality from Lemma \ref{gain_n} with the $L^{5/3}L^{5/3}$ estimate from Lemma \ref{LpLp}.

\begin{lemma}\label{a_inf}
For any $0<t<T$ and any interger $n$ there exists a constant $C(p,T,u_0,n)$ such that for $\alpha = \frac{(n+1)}{(3n+2)}$:
\begin{align*}
\|a[u]\|_{L^\infty(t,T, \mathbb{R}{^3})}\le C(T,u_0,n)\left(\frac{1}{t}+ 1\right)^{{\alpha}}.
\end{align*}
\end{lemma}
\begin{proof}
Let $r>0$; for $p>3/2$ we have 
\begin{align*}
4\pi a[u](x,t) &= \int_{B_r(x)} \frac{u(y)}{|x-y|}\;dy + \int_{B_r^c(x)} \frac{u(y)}{|x-y|}\;dy \\
&\le \frac{1}{r} \|u\|_{L^\infty(L^1)} + 4\pi \|u\|_{L^\infty(L^p)}r^{2-3/p},
\end{align*}
applying H\"older inequality. The minimum of the function $
F(r) = \frac{c_1}{r} + c_2  r^{2-3/p}
$
is reached at the point 
$$
r_{min} = \left( \frac{c_1}{\left(2-3/p\right)c_2}\right)^{p/(3(p-1))}
$$ 
and this implies 
\begin{align*}
a[u](x,t) \le 4 \|u\|_{L^\infty(L^1)} ^{\frac{2p-3}{3(p-1)}} \|u\|_{L^\infty(L^p)}^{\frac{p}{3(p-1)}}.
\end{align*}

From Lemma \ref{gain_n} we know that 
\begin{align*}
   \sup \limits_{T/4 \leq t\leq T} \left \{ \int  u^{p+n}(t)\;dx\right\}   \le  2^{n(n+1)}C(p)^n\left(\frac{1}{T}+ 1\right)^{n+1}\int_{0}^{T} \int u^p(t)\;dxdt,
    \end{align*}	  
and taking $p=5/3$ and using Lemma \ref{LpLp} we get 
\begin{align}\label{u_inf_n}
 \| u\|_{L^\infty( T/4, T,L^{5/3+n}(\mathbb{R}^3))}   \le  C(n,T,u_0) \left(\frac{1}{T}+ 1\right)^{\frac{n+1}{5/3+n}}.
    \end{align}	
Going back to $a[u]$ this last estimate implies 
\begin{align}\label{a_inf_n}
  \sup \limits_{t\in (T/4,T),x\in \mathbb{R}^3} a[u](x,t) \le& c(u_0) \|u\|_{L^\infty(T/4,T; L^{5/3+n})}^{\frac{5/3+n}{2+3n}}\nonumber \\
  \le& C(n,T,u_0) \left(\frac{1}{T}+ 1\right)^{\frac{n+1}{3n+2}}.
\end{align}

\end{proof}

\subsection {De-Giorgi iteration and $L^\infty$-regularization}

\begin{prop}\label{prop_iter}
Let $p=\frac{5}{3}$ and $q$ as in Lemma \ref{lem:Inequalities Sobolev weight aII}. We have 
\begin{align*}
   \sup \limits_{T_{n+1} \leq t\leq T} \left \{ \int (\eta_n u_n^{p/2})^2\;dx\right\} &+ \frac{(p-1)}{4p}  \int_{T_{n+1}}^{T}\int a|\nabla (\eta_n u_n^{p/2})|^2\;dxdt\\
      \le  & C_0 \int_{T_n}^{T}\int   a  (\eta_{n-1}u_{n-1}^{p/2})^q\;dxdt,
 \end{align*}
 with 
 $$
 C_0:=C^{n-1} C(R,p)\left(\frac{1}{T} + 1\right)\left(\frac{1}{M}\right)^{\frac{p(q-2)}{2}-1} .
 $$
\end{prop}

\begin{proof}
Consider the sequence of times and radii
  \begin{align*}
    T_n & = \frac{1}{4}\left ( 2- \frac{1}{2^n} \right )T,\quad R_n  = \frac{1}{2}\left(1 + \frac{1}{2^n}\right)R,
  \end{align*}
  and, for every $n \ge 1$, let $B_n$ denote the ball $ B_n := B_{R_n}(0)$. 
 
  Let $\eta_n$ be a $C^\infty$ function supported in $B_n$, with $0\leq \eta_n\leq 1$ everywhere,  $\eta_n=1$ in $B_{n+1}$,  $\|\nabla \eta_n\|_\infty \leq C \eta_n2^{n+1}$ and $\|D^2(\eta_n)\|_\infty \leq  C 2^{2n+2}$. Corollary \ref{cor_iterat} says that for $k_n:=M\left(1-\frac{1}{2^n}\right)$, $T_1 = T_n$, $T_2 = T_{n+1}$, $T_3 = T$,  $T_{n+1}-T_{n} = \frac{T}{2^{n+1}}$ and  
$$
u_n:=\left(u-M\left(1-\frac{1}{2^n}\right)\right)_+
$$
we have 
\begin{align*}
   \sup \limits_{T_{n+1} \leq t\leq T} \left \{ \int \eta_n^2 u_n^p(t)\;dx\right\} &+ \frac{(p-1)}{4p}  \int_{T_{n+1}}^{T}\int a|\nabla (\eta_n u_n^{p/2})|^2\;dxdt\\
      \le  &\; \left(\frac{2^{n+2}}{T}+C(\varepsilon, p)\right) \int_{T_n}^{T} \int \eta_n^2u_n^p\;dxdt \\
      &+C(p) \int_{T_n}^{T}\int  u_n^p (a\nabla \eta_n,\nabla \eta_n) \;dxdt+2 p k_n^2 \int_{T_n}^{T}\int \eta_n^2 u_n^{p-1}\;dxdt\\
     & + \int_{T_n}^{T}\int a u_n^p \eta_n |\Delta \eta_n|\;dxdt \le  \; U_n,
    \end{align*}	
with
\begin{align*}
U_n:= &\left(\frac{2^{n+2}}{T}+C(\varepsilon, p)\right) \int_{T_n}^{T} \int \eta_n^2u_n^p\;dxdt \\
      &\;+(C(p)+1)2^{2n+2} \int_{T_n}^{T}\int_{B_n}  a  \eta_n^2 u_n^p  \;dxdt +2 p k_n^2 \int_{T_n}^{T}\int \eta_n^2 u_n^{p-1}\;dxdt.   	\end{align*}
We start by estimating the last term of $U_n$: since $\eta_{n-1}=1$ on $B_n$ and $\chi_{\{u_{n}\ge 0\}} = \chi_{\{u_{n-1}\ge \frac{M}{2^{n}}\}}$ we have 
\begin{align*}
2 p k_n^2 \int_{T_n}^{T}\int \eta_n^2 u_n^{p-1}\;dxdt \le &\; 2 p M^2 \int_{T_n}^{T}\int_{B_n}  u_n^{p-1}\;dxdt \\
=  &\; 2 p M^2 \int_{T_n}^{T}\int_{B_n}  u_n^{p-1}\chi_{\{u_{n-1}\ge \frac{M}{2^{n}}\}}\;dxdt\\
\le & \;2 p M^2 \int_{T_n}^{T}\int_{B_n}  u_{n-1}^{p-1}\chi_{\{\eta_{n-1}^{2/p}u_{n-1}\ge \frac{M}{2^{n}}\}}\;dxdt.
\end{align*}
H\"older inequality yields 
\begin{align*}
2 p k_n^2 \int_{T_n}^{T}\int \eta_n^2 u_n^{p-1}\;dxdt \le &  \;2 p M^2 \int_{T_n}^{T}\left( \int_{B_n} u_{n-1}^{\frac{pq}{2}}\;dx\right)^{\frac{2(p-1)}{pq}} \cdot\\
&\quad\quad \quad \quad \quad  \;\cdot \left( \int_{B_n} \chi_{\{\eta_{n-1}^{2/p}u_{n-1}\ge \frac{M}{2^{n}}\}}\;dx\right)^{\frac{pq-2(p-1)}{pq}}\;dt.
\end{align*}
Using Chebyshev's inequality 
$$
\int_{B_n} \chi_{\{\eta_{n-1}^{2/p}u_{n-1}\ge \frac{M}{2^{n}}\}}\;dx \le \left(\frac{2^{n}}{M}\right)^{pq/2} \int (\eta_{n-1}^{2/p}u_{n-1})^{pq/2}\;dx 
$$
 we get 
\begin{align*}
2 p k_n^2 \int_{T_n}^{T}\int \eta_n^2 u_n^{p-1}\;dxdt \le   &\;2 p M^2\left(\frac{2^{n}}{M}\right)^{\frac{pq-2(p-1)}{2}}  \int_{T_n}^{T}\left( \int_{B_n} u_{n-1}^{\frac{pq}{2}}\;dx\right)^{\frac{2(p-1)}{pq}} \left(\int (\eta_{n-1}u_{n-1}^{p/2})^q\;dx\right)^{\frac{pq-2(p-1)}{pq}}\;dt \\
=  &\;2 p M^2\left(\frac{2^{n}}{M}\right)^{\frac{pq-2(p-1)}{2}}  \int_{T_n}^{T}\left( \int_{B_n} \eta_{n-1}^qu_{n-1}^{\frac{pq}{2}}\;dx\right)^{\frac{2(p-1)}{pq}} \cdot \\
&  \quad \quad \quad \quad \quad \quad \quad \quad \quad  \quad \quad \quad\cdot \left(\int (\eta_{n-1}u_{n-1}^{p/2})^q\;dx\right)^{\frac{pq-2(p-1)}{pq}}\;dt\\
=  &\;2 p M^2\left(\frac{2^{n}}{M}\right)^{\frac{pq-2(p-1)}{2}}  \int_{T_n}^{T}\int(\eta_{n-1}u_{n-1}^{p/2})^q\;dxdt\\
&{{ \le  \;2 p C(R) M^2\left(\frac{2^{n}}{M}\right)^{\frac{pq-2(p-1)}{2}}  \int_{T_n}^{T}\int a(\eta_{n-1}u_{n-1}^{p/2})^q\;dxdt.}}
\end{align*}
We now estimate the first two terms of $U_n$: 
\begin{align*}
&\left(\frac{2^{n+2}}{T}+C(\varepsilon, p)\right) \int_{T_n}^{T} \int \eta_n^2u_n^p\;dxdt +(C(p)+1)2^{2n+2} \int_{T_n}^{T}\int_{B_n}  a  \eta_n^2 u_n^p  \;dxdt \\
&\le 2^{2n+2} \left(\frac{1}{T} + {{C(p, R)}}\right) \int_{T_n}^{T}\int_{B_n}  a  \eta_n^2 u_n^p  \;dxdt \\
&\le 2^{2n+2} \left(\frac{1}{T} + {{C(p, R)}}\right) \int_{T_n}^{T}\int_{B_n}  a  u_{n-1}^p\chi_{\{u_n\ge 0\}}  \;dxdt \\
&\le 2^{2n+2} \left(\frac{1}{T} + {{C(p, R)}}\right) \int_{T_n}^{T}\int  a  \eta^{2}_{n-1}u_{n-1}^p\chi_{\{u_{n-1}\ge \frac{M}{2^{n}}\}}  \;dxdt.
      \end{align*}
Similarly as before, we apply H\"older's  and Chebyshev's inequalities and obtain 
\begin{align*}
\int a   \eta^{2}_{n-1} u_{n-1}^p\chi_{\{u_{n-1}\ge \frac{M}{2^{n}}\}}  \;dx & \le \left( \int  a  \eta^{q}_{n-1} u_{n-1}^{pq/2}\;dx \right)^{2/q}\left( \int a\chi_{\{\eta_{n-1}^{2/p}u_{n-1}\ge \frac{M}{2^{n+1}}\}}  \;dx\right)^{(q-2)/q} \\
& \le \left( \int  a  (\eta_{n-1}  u_{n-1}^{p/2})^{q}\;dx \right)^{2/q}\left( \left(\frac{2^{n}}{M}\right)^{pq/2}\int a \eta_{n-1}^{q} u_{n-1}^{pq/2} \;dx \right)^{(q-2)/q} \\
& = \left(\frac{2^{n}}{M}\right)^{p(q-2)/2} \int   a  (\eta_{n-1}u_{n-1}^{p/2})^q\;dx ,
\end{align*}
which implies 
\begin{align*}
&\left(\frac{2^{n+2}}{T}+C(\varepsilon, p)\right) \int_{T_n}^{T} \int \eta_n^2u_n^p\;dxdt +(C(p)+1)2^{2n+2} \int_{T_n}^{T}\int  a  \eta_n^2 u_n^p  \;dxdt \\
 & \le \;2^{2n+2} \left(\frac{1}{T} + {{C(p, R)}}\right) \left(\frac{2^{n+1}}{M}\right)^{p(q-2)/2} \int_{T_n}^{T}\int   a  (\eta_{n-1}u_{n-1}^{p/2})^q\;dx. 
 \end{align*}
 Summarizing we obtain: 
 \begin{align*}
 U_n \le& \left(2 p C(R) M^2\left(\frac{2^{n+1}}{M}\right)^{\frac{pq-2(p-1)}{2}} + 2^{2n+2} \left(\frac{1}{T} + {{C(p, R)}}\right) \left(\frac{2^{n}}{M}\right)^{\frac{p(q-2)}{2}}\right) \int_{T_n}^{T}\int   a  (\eta_{n-1}u_{n-1}^{p/2})^q\;dxdt\\
 \le & 4^{n-1} C(R,p)\left(\frac{1}{T} + 1\right)\left(\frac{1}{M}\right)^{\frac{p(q-2)}{2}-1} \int_{T_n}^{T}\int   a  (\eta_{n-1}u_{n-1}^{p/2})^q\;dxdt.
 \end{align*}
 This completes the proof.
 
 
\end{proof}

\begin{prop}
Let $T>0$ and $R>0$. Given any $s>1$ there exists a constant that only depends on $s$, $R$, the mass and second moment of $u$ (hence on $T$) such that 
$$
\sup_{(T/4,T) \times B_{R/2}} u(x,t) \le c_0(s,R,T)  \left(\frac{1}{T}+1\right)^{s}.
$$

\end{prop}

\begin{proof}
Lemma \ref{lem:Inequalities Sobolev weight aII} for $\phi =\eta_n u_n^{p/2}$ implies 
\begin{align}\label{Sob_a}
\left( \int_{T_{n+1}}^T \int a (\eta_n u_n^{p/2})^q\;dxdt\right)^{2/q}\le   & \sup \limits_{T_{n+1} \leq t\leq T} \left \{ \int (\eta_n u_n^{p/2})^2\;dx\right\} \\
&\;+ \frac{(p-1)}{4p}  \int_{T_{n+1}}^{T}\int a|\nabla (\eta_n u_n^{p/2})|^2\;dxdt. \nonumber 
\end{align}
   Then Proposition \ref{prop_iter} says that 
   
   \begin{align*}
   \sup \limits_{T_{n+1} \leq t\leq T} \left \{ \int (\eta_n u_n^{p/2})^2\;dx\right\} &+ \frac{(p-1)}{4p}  \int_{T_{n+1}}^{T}\int a|\nabla (\eta_n u_n^{p/2})|^2\;dxdt\\
      \le & \;U_n \le   C_{n,p,T,M} \int_{T_n}^{T}\int   a  (\eta_{n-1}u_{n-1}^{p/2})^q\;dxdt\\
      \le &  C_{n,p,T,M}\left(\sup \limits_{T_{n} \leq t\leq T} \left \{ \int (\eta_{n-1} u_{n-1}^{p/2})^2\;dx\right\} + \frac{(p-1)}{4p}  \int_{T_{n}}^{T}\int a|\nabla (\eta_{n-1} u_{n-1}^{p/2})|^2\;dxdt \right)^{\frac{q}{2}} \\
      \le &\; C_{n,p,T,M}\; U_{n-1}^{\frac{q}{2}}, 
 \end{align*}
 with  
 $$
 C_{n,p,T,M}:=4^{n-1}\underbrace{C(p, R)\left(\frac{1}{T} +1\right)\left(\frac{1}{M}\right)^{\frac{p(q-2)}{2}-1}}_{ :=C_{p,R,T,M}}.
 $$
 This leads to a recurrence relation 
 $$
 U_n\le 4^{n-1}C_{p,R,T,M}U_{n-1}^{\frac{q}{2}}.
   $$
   A standard induction argument shows that the above recurrence relation yields
   \begin{align}\label{end_sum}
   \lim_{n\to +\infty} U_n =0,
   \end{align}
provided the initial step 
\begin{align*}
U_0:= &\left(\frac{1}{T}+C(\varepsilon, p)\right) \int_{T_0}^{T} \int \eta_0^2u^p+  a  \eta_0^2 u^p  \;dxdt, \quad T_0=T/4,\\
\end{align*}   
is small enough. For completeness we sketch this last argument: assume for a certain $n\ge 0$ 
\begin{align}\label{ind_n}
4^n U_n^{\frac{q}{2}-1}\le \frac{1}{C_{p,R,T,M} (8)^{\frac{1}{\frac{q}{2}-1}}},
\end{align}
we show that the same is true for $n+1$: using (\ref{ind_n}) we get 
\begin{align*}
4^{n+1} U_{n+1}^{\frac{q}{2}-1}& \le  4^{n+1} \left( 4^{n}C_{p,R,T,M}U_{n}^{\frac{q}{2}}\right)^{\frac{q}{2}-1} \le 4 C_{p,R,T,M}^{\frac{q}{2}-1}  \left( C^{n}U_{n}^{\frac{q}{2}-1} \right)^{\frac{q}{2}} \\
&\le 4 C_{p,R,T,M}^{\frac{q}{2}-1}  \left( \frac{1}{C_{p,R,T,M} (2C)^{\frac{1}{\frac{q}{2}-1}}}\right)^{\frac{q}{2}} \\
&\le  C_{p,R,T,M}^{-1} \frac{4}{(8)^{\frac{\frac{q}{2}}{\frac{q}{2}-1}}} \le \frac{1}{C_{p,R,T,M} (8)^{\frac{1}{\frac{q}{2}-1}}}.
\end{align*}
Therefore if (\ref{ind_n}) holds for $U_0$, i.e.
\begin{align}\label{init_cond_small}
U_0^{\frac{q}{2}-1}\le \frac{1}{C_{p,R,T,M} (8)^{\frac{1}{\frac{q}{2}-1}}},
\end{align}
then 
$$
\lim_{n\to +\infty} U_{n+1}^{\frac{q}{2}-1} \le \lim_{n\to +\infty} \frac{c}{4^{n}}  =0,
$$
and (\ref{end_sum}) is proven. 

We are left to prove that for $M$ big enough the condition (\ref{init_cond_small}) is satisfied. Let $p = 5/3 + n$ with $n$ any positive integer. Inequalities (\ref{u_inf_n}) and (\ref{a_inf_n}) imply
\begin{align*}
U_0 & \le c(n)\left(\frac{1}{T}+1\right) \int_{T/4}^{T} \int u^{5/3+n}+  a u^{5/3+n}  \;dxdt \\
& \le   c(n) \left(\frac{1}{T}+1\right) \left( \|a\|_{L^\infty((T/4,T)\times \mathbb{R}^3)} +1\right) \int_{T/4}^T \int u^{5/3+n} \;dxdt\\
& \le  c(n,u_0,T) \left(\frac{1}{T}+1\right)^{1+ \frac{n+1}{3n+2} + n+1}  = c(n,u_0,T) \left(\frac{1}{T}+1\right)^{ \frac{7n+5}{3n+2} + n}.
\end{align*}
We chose $M$ big enough so that 
$$
c(n) \left(\frac{1}{T}+1\right)^{ \left(\frac{7n+5}{3n+2} + n\right)\left(\frac{q}{2}-1\right)}\left(\frac{1}{T} +1\right)\left(\frac{1}{M}\right)^{\frac{(5/3+n)(q-2)}{2}-1} \le \frac{1}{8^{\frac{1}{\frac{q}{2}-1}}}.
 $$
 or equivalently 
 $$
 M> c(n) \left(\frac{1}{T}+1\right) ^{\alpha(n)}
 $$
 with 
 $$
\alpha(n) = \frac{ \left(\frac{7n+5}{3n+2} + n\right)\left(\frac{q}{2}-1\right)}{{(5/3+n)(\frac{q}{2}-1)}-1}.
 $$
 Note that $\alpha(n) \ge 0$ for each $n\ge 0$ and $\alpha(n)\to 0$ as $n\to +\infty$. Therefore given any $s>1$ there exists an integer $n$ such that $\alpha(n)<s$ and this concludes the proof.

    
\end{proof}


\begin{thebibliography}{15}

\bibitem{AleVil2004}
R. Alexandre and C. Villani. {\em On the Landau approximation in plasma physics.}
Ann. Inst. Henri Poincar\'e, C Anal. Non Lin\'eaire 21 (1) (2004) 61-95.

\bibitem{CWW}
S.-Y. A. Chang, J.M. Wilson and T.H. Wolff.
{\em Some weighted norm inequalities concerning the Schr\"odinger operators. }
Comment. Math. Helv. 60 (1985), no. 2, 217-246. 

\bibitem{CW1}
 S. Chanillo and R. Wheeden. {\em L-p estimates for fractional integrals and Sobolev inequalities with applications to Schr\"odinger operators.} Communications in partial differential equations, 10(9):1077 - 1116, 1985.
 
 \bibitem{CW2}
S. Chanillo and R. Wheeden. {\em Weighted Poincar\'e and Sobolev inequalities and estimates for weighted Peano maximal functions. } American Journal of Mathematics, 107(5):1191-1226, 1985.

\bibitem{Des2015}
L. Desvillettes. {\em Entropy dissipation estimates for the Landau equation in the Coulomb case and applications.}
Journal of Functional Analysis 269 (2015) 1359 - 1403.

\bibitem{F83} 
C. Fefferman. {\em The uncertainty principle}. Bull. Amer. Math. Soc. 9, (1983) 129-206.

\bibitem{FP78}
C. Fefferman and D. H. Phong. {\em On positivity of pseudo-differential operators.} Proc. Nat. Acad. Sci. U.S.A. 75 (1978), no. 10, 4673 - 4674. 

\bibitem{GK}
Y. Giga and R.V. Kohn. {\em Asymptotically self-similar blow-up of semilinear heat equations. } Communications on Pure and Applied Mathematics, 38(3):297 - 319, 1985.

\bibitem{GG}
M. Gualdani and N. Guillen. {\em Estimates for radial solutions of the homogeneous Landau equation with Coulomb potential. } Analysis and PDE, 9(8):1772 - 1809, 2016.

\bibitem{GG17}
M. Gualdani and N. Guillen. {\em On $A_p$ weights and the homogeneous Landau equation}. Under review. 

\bibitem{GZ} 
M. Gualdani and N. Zamponi. {\em Global existence of weak even solutions for an isotropic Landau equation with Coulomb potential}. Under review. 

\bibitem{GKS}
P. Gressman, J. Krieger, and R. Strain. {\em A non-local inequality and global existence. }Advances in Mathematics, 230(2):642 - 648, 2012.

\bibitem{KS}
J. Krieger and R.Strain. {\em Global solutions to a non-local diffusion equation with quadratic nonlinearity.} Comm. Partial Differential Equations, 37(4):647-689, 2012.

\bibitem{Jue2015}
A.~J\"ungel. {\em The boundedness-by-entropy method for cross-diffusion systems.} Nonlinearity 28.6 (2015), 1963.

\bibitem{Jue2016}
A.~J\"ungel. {\em Entropy methods for diffusive partial differential equations.} {Springer}, 2016.

\bibitem{SW}
E. Sawyer and R. Wheeden. {\em Weighted inequalities for fractional integrals on Euclidean and homogeneous spaces.} American Journal of Mathematics, 114(4):813- 874, 1992.

\bibitem{Vil1998}
C. Villani. {\em On a new class of weak solutions to the spatially homogeneous Boltzmann and Landau equations.}
Arch. Ration. Mech. Anal. 143 (3) (1998) 273-307.

\bibitem{Zei90}
E. Zeidler. {\em Nonlinear functional analysis and its applications}, vol. II/B (1990).

\end{thebibliography}
\end{document}